\documentclass{amsart}

\usepackage{amsmath, amssymb, diagrams}
\usepackage{hyperref}

\usepackage[OT2,OT1]{fontenc} 
\DeclareSymbolFont{cyrletters}{OT2}{wncyr}{m}{n}
\DeclareMathSymbol{\Sha}{\mathalpha}{cyrletters}{"58}

\newcommand{\QQ}{\mathbb Q}
\newcommand{\ZZ}{\mathbb Z}

\DeclareMathOperator{\Sel}{Sel}
\DeclareMathOperator{\coker}{coker}

\DeclareMathOperator{\Gal}{Gal}
\DeclareMathOperator{\ord}{ord}
\DeclareMathOperator{\cd}{cd}
\DeclareMathOperator{\image}{Im}
\DeclareMathOperator{\Ind}{Ind}

\newcommand{\Fc}{F^{\rm cyc}}
\newcommand{\Fiw}{F_{\infty,w}}
\newcommand{\tEp}{\tilde{E}_{v,p^\infty}}
\newcommand{\Ep}{E_{p^\infty}}
\newcommand{\Sf}{\Sel(E/F_\infty)}
\newcommand{\Sc}{\Sel(E/\Fc)}
\newcommand{\Sp}{\Sel'(E/\Fc)}

\newcommand{\MHS}{{\mathfrak M}_H(\Sigma)}

\newtheorem{thm}{Theorem}[section]
\newtheorem{prop}[thm]{Proposition}
\newtheorem{lem}[thm]{Lemma}
\newtheorem{cor}[thm]{Corollary}

\theoremstyle{definition}
\newtheorem{definition}[thm]{Definition}

\newtheorem*{remark}{Remark}
\newtheorem*{convention}{Convention}

\begin{document}  

\title{Akashi series of Selmer groups}

\author{Sarah Livia Zerbes}
\address{Department of Mathematics, Harrison Building, University of Exeter, Exeter EX4 4QF, UK}
\email{s.zerbes@exeter.ac.uk}

\thanks{Supported by EPSRC Postdoctoral Fellowship EP/F043007/1.}

\begin{abstract}
 We study the Selmer group of an elliptic curve over an admissible $p$-adic Lie extension of a number field $F$. We give a formula for the Akashi series attached to this module, in terms of the corresponding objects for the cyclotomic $\ZZ_p$-extension and certain correction terms. This extends our earlier work \cite{zerbes09}, in particular since it applies to elliptic curves having split multiplicative reduction at some primes above $p$, in which case the Akashi series can have additional zeros. 
\end{abstract}

\maketitle

\section{Introduction}

 \subsection{Background}

  Let $p$ be a prime number, $F$ a number field, and $F_\infty / F$ an infinite Galois extension whose Galois group $\Sigma = \Gal(F_\infty / F)$ is a $p$-adic Lie group. We will always assume that $F_\infty / F$ is an {\it admissible} $p$-adic Lie extension, by which we mean that it satisfies the following conditions:
  \begin{enumerate}
   \item $F_\infty$ contains the cyclotomic $\ZZ_p$-extension $\Fc$ of $F$,
   \item $F_\infty$ is unramified outside a finite set of primes of $F$,
   \item $\Sigma$ has no element of order $p$. 
  \end{enumerate}
  
  This class of extensions has been much studied in the context of Iwasawa theory, following the seminal paper \cite{cfksv}. We recall some of the general setup. Let $H$ denote the normal subgroup $\Gal(F_\infty / \Fc)$ of $\Sigma$, and $\Gamma = \Gal(\Fc/F) = \Sigma / H$. 

  \begin{definition}
   We define $\MHS$ to be the full subcategory of the category of finitely-generated $\Lambda(\Sigma)$-modules consisting of objects $M$ for which $M / M(p)$ is finitely-generated over $\Lambda(H)$, where $M(p)$ is the $p$-torsion part of $M$.
  \end{definition}

  If $E$ is an elliptic curve over $F$, then the Selmer group $\Sf$ is a discrete $p$-primary abelian group with an action of $\Sigma$, and its Pontryagin dual $\mathcal{C}(E / F_\infty)$ is known to be finitely generated as a $\Lambda(\Sigma)$-module \cite[theorem 2.7]{coates99}. Moreover, it is conjectured that if $p \ge 5$ and $E$ has good ordinary or split multiplicative reduction at all primes $p$, we should have $\mathcal{C}(E / F_\infty) \in \MHS$.

  The natural map $\Lambda(\Sigma) \to \Lambda(\Gamma)$ induces a map $K_0(\Lambda(\Sigma)) \to K_0(\Lambda(\Gamma))$, which maps the class of a module $M$ to the alternating sum of its homology groups (which is finite as $H$ has finite $p$-cohomological dimension):
  \[[M] \mapsto \sum_{i \ge 0} (-1)^i [H_i(H, M)].\]
  For $M \in \MHS$, these homology groups are torsion $\Lambda(G)$-modules, so we obtain a map $K_0(\MHS) \to K_0(\Lambda(\Gamma))^{\rm tors}$; since the class of a torsion $\Lambda(\Gamma)$-module is determined by its characteristic power series, the latter is isomorphic to $Q(\Gamma)^\times / \Lambda(\Gamma)^\times$, where $Q(\Gamma)$ is the fraction field of $\Lambda(\Gamma)$. We define the {\it Akashi series} $f^\Sigma_M$ to be the image of $M$ under this map.

  In earlier work \cite{zerbes04,zerbes09}, motivated by applications to generalised Euler characteristics, we derived formulae for the $\Sigma$-Akashi series of $\mathcal{C}(E / F_\infty)$ in terms of the corresponding objects over $\Fc$ and various explicit correction terms. The purpose of this paper is to study these Akashi series in their own right, without imposing the fairly restrictive conditions (such as finiteness of the $p$-primary part of the Shafarevich-Tate group) required to ensure that the generalised Euler characteristic is defined. This allows us (for instance) to consider elliptic curves with split multiplicative reduction at some primes above $p$, in which case the Akashi series can have additional zeros at $T = 0$.

  If $E$ has good ordinary reduction at all primes above $p$, the main conjecture of noncommutative Iwasawa theory for $E / F_\infty$, formulated in \cite[\S 5]{cfksv}, predicts the existence of an element $\xi$ in $K_1(\Lambda(\Sigma)_{S^*})$ satisfying $\partial_\Sigma (\xi) = [\mathcal{C}(E/F_\infty)]$, and whose values at Artin characters of $\Sigma$ coincide (up to an appropriate correction factor) with the values at $s = 1$ of the $L$-functions of $E$ and its twists. Our results are consistent with this conjecture in the good ordinary case; in the split multiplicative case they would be consistent with an analogue of this conjecture taking into account the phenomenon of additional zeros.

 \subsection{Statement of the main results}

  In \cite{zerbes09}, we considered the Selmer group $\Sf$ for an elliptic curve $E/F$ over an admissible $p$-adic Lie 
  extension $F_\infty/F$, relating $\Sf$ to the cyclotomic Selmer group $\Sc$ and certain ``correction factors'' $J_v(\Fc)$ (whose definition we shall recall in the next section), for primes $v$ lying in the finite set $S'$ of primes of $F$ not dividing $p$ whose inertia group in $\Sigma$ is infinite. 

  In the course of the proof of \cite[Theorem 1.1]{zerbes09}, we showed that if $E$ satisfies the hypotheses
  \begin{itemize}
   \item $\Sha(E/F)(p)$ is finite,
   \item $E$ has good ordinary reduction at all primes of $F$ dividing $p$,
   \item $\mathcal{C} (E/F_\infty)\in\MHS$,
   \item (Fin$_{\rm glob}$) $H^i(H,\Ep(F_\infty))$ is finite for all $i\geq 0$,
   \item (Fin$_{\rm loc}$) $H^i(H_w,\tEp(k_{\infty,w}))$ is finite for all $i\geq 0$ and for every prime $v$ of $F$ dividing $p$, where $\tilde E_v$ denotes the reduced curve and $k_{\infty, w}$ is the residue field of a prime of $F_\infty$ above $v$,
  \end{itemize}
  then the $\Sigma$-Akashi series of $\Sf$ is related to the $\Gamma$-Akashi series of $\Sc$ (i.e.~its characteristic element) by the following equation:
  \begin{equation}\label{relationakashi3}
   f^\Sigma_{\Sf}=f^\Gamma_{\Sc}\times\prod_{v\in S'}f^\Gamma_{J_v(\Fc)} \pmod{\Lambda(\Gamma)^\times}.
  \end{equation}
  
  It seems difficult to verify the finiteness of $\Sha(E/F)(p)$ and the conditions (Fin$_{\rm glob}$) and (Fin$_{\rm loc}$) in general. In the present paper we relax the condition on $E$, and replace it with a slightly stronger assumption on the $p$-adic Lie extension $F_\infty/F$:

  \begin{definition}
   A $p$-adic Lie extension $F_\infty$ of $F$ is {\it strongly admissible} if it is admissible, and for each prime $v$ of $F$ dividing $p$, the local extension $\Fiw$ of $F_v$ (for $w$ a prime of $F_\infty$ above $v$) contains the unramified $\ZZ_p$-extension of $F_v$.
  \end{definition}

  If $F_\infty/F$ is strongly admissible in this sense then we can prove a relation between the Akashi series of $\Sf$ and $\Sc$, analogous to \eqref{relationakashi3}:
  
  \begin{thm}\label{maintheorem}
   Let $F$ be a finite extension of $\QQ$, $E$ an elliptic curve defined over $F$ and $p$ a prime $\geq 5$. Let $F_\infty$ be a strongly admissible $p$-adic Lie extension of $F$ with Galois group $\Sigma=\Gal(F_\infty/F)$. Assume that 
   \begin{enumerate}
    \item $E$ has either good ordinary or split multiplicative reduction at each prime of $F$ dividing $p$,
    \item $\mathcal{C} (E/F_\infty)\in\MHS$.
   \end{enumerate}
   Then 
   \[ f^\Sigma_{\Sf}=T^r \times f^\Gamma_{\Sc}\times\prod_{v \in S'}f^\Gamma_{J_v(\Fc)} \pmod{\Lambda(\Gamma)^\times},\] 
   where $r$ is the number of primes of $F$ dividing $p$ where $E$ has split multiplicative reduction,
   $S'$ is the set of primes of $F$ not dividing $p$ whose inertia group in $\Sigma$ is infinite and
   $f_v^\Gamma$ is the $\Gamma$-Akashi series of $J_v(\Fc)$. 
  \end{thm}
  
  \begin{remark} It follows from \cite[lemma 2.14]{zerbes09} that $f^\Gamma_{J_v(\Fc)}$ has nonzero constant term equal to the $L$-factor $L_v(E,1) \pmod{\ZZ_p^\times}$, so we have
  \[
   \ord_{T=0}(f_{\Sf}) = \ord_{T=0}(f_{\Sc})+r,
  \]
  i.e.~the Akashi series has extra zeroes, and its leading term $\alpha^\Sigma_{\Sf}$ is given by 
  \[ \alpha^\Sigma_{\Sf} = \alpha_{\Sc}^\Gamma\times\prod_{v \in S'}L_v(E,1)\pmod{\ZZ_p^\times}.\]
  In particular, {\it if} the modules $\Sf$ and $\Sc$ posess generalised Euler characteristics in the sense of \cite[definition 2.7]{zerbes09}, which is equivalent to being ``semisimple at $T = 0$'' in the sense of \cite[definition 3.11]{burnsvenjakob06}, then the generalised Euler characteristics of $\Sf$ and $\Sc$ are related by the same formula as derived in \cite{zerbes09}.
  \end{remark}
  
  The proof of Theorem \ref{maintheorem} relies on certain symmetry properties of the Hochschild-Serre spectral sequence, which imply that for a strongly admissible $p$-adic Lie extension the Akashi series of every finitely generated $\ZZ_p$-module is equal to $1$ (c.f. Proposition \ref{akashi=1}). This observation should be considered as a substitute for the conditions (Fin$_{glob}$) and (Fin$_{loc}$) appearing in our earlier work. Since Proposition \ref{akashi=1} is independent of the action of $\Sigma$, it also allows us to deal with the case when $E$ has split multiplicative reduction at some primes of $F$ dividing $p$. 

  Finally, in section \ref{gl2ext} we consider an important special case, when $F_\infty$ is the field $F(\Ep)$ of $p$-power torsion points of $E$. We show that this extension is strongly admissible if and only if $E$ has good reduction at all primes above $p$, and when this is not the case, we explicitly calculate the necessary correction terms to give a formula for the Akashi series of $\Sf$.
  

 \section{Akashi series}
 
  In this section, we prove some results about the Akashi series of elements in $\MHS$. Let us begin by recalling the definition of the Akashi series and some of its properties.
  
  \begin{definition}
   Let $M\in\MHS$. Then the homology group $H_i(H,M)$ is a finitely generated torsion $\Lambda(\Gamma)$-module for all $i\geq 0$. Let $f^\Gamma_{H_i(H,M)}$ be the characteristic element of the $\Lambda(\Gamma)$-module $H_i(H,M)$, and define the Akashi series of $M$ to be
   \[
    f^\Sigma_M=\prod_{i\geq 0}\big(f^\Gamma_{H_i(H,M)}\big)^{(-1)^i}.
   \]
  \end{definition}

  \begin{convention}
   Let $X$ be a discrete $p$-primary $\Sigma$-module such that $M=X^\vee\in\MHS$. Define the Akashi series $f_X$ of $X$ to be the Akashi series of $M$.
  \end{convention}
 
  \begin{prop}\label{multiplicativity3}
   If we have an exact sequence of modules in $\MHS$
   \[
    0\rTo L\rTo M\rTo N\rTo 0,
   \]
   then $f^\Sigma_M=f^\Sigma_Nf^\Sigma_L$.
  \end{prop}

  \begin{proof}
   Clear from the definition.
  \end{proof}
  
  The next result is an extension of Lemma 4.5 in \cite{coatesschneidersujatha03}:
  let $F$ be a finite extension of $\QQ$ and $p$ a prime $\geq 5$. Let $F_\infty$ be an
  admissible $p$-adic Lie extension of $F$ with Galois group $\Sigma=\Gal(F_\infty/F)$.
  
  \begin{prop}\label{akashi=1}
   Assume that there exists a subextension $K_\infty$ of $F_\infty$ containing $\Fc$ such that the Galois
   group $\Sigma' = \Gal(K_\infty/F)$ can be written as a direct product $\Gamma \times G$, where $\Gamma \cong\Gal(\Fc/F)$ and $G$
   has dimension $\geq 1$. Let $M$ be a $\Sigma$-module which is cofinitely generated as a $\ZZ_p$-module.
   Then $f^\Sigma_M\thicksim 1$.
  \end{prop}
  \begin{proof}
   Let $R=\Gal(F_\infty/K_\infty)$. By the Hochschild-Serre spectral sequence, we have
   \[ f^\Sigma_{M} = \prod_{i \ge 0} \left(f^{\Sigma'}_{H^i(R, M)}\right)^{(-1)^i},\]
   and the $\Lambda(\Sigma')$-modules $H^i(R, M)$ are also $\ZZ_p$-cofinitely generated. Hence we may assume without loss of generality that $F_\infty = K_\infty$, which is Lemma 4.5 in \cite{coatesschneidersujatha03}.
  \end{proof}

  \begin{lem}\label{chareltinduced}
   If $\Gamma' \subseteq \Gamma$ is a finite-index subgroup, and $M$ is a torsion $\Lambda(\Gamma')$-module, then the characteristic element of the $\Lambda(\Gamma)$-module $N = \Ind_{\Gamma'}^\Gamma M = \Lambda(\Gamma) \otimes_{\Lambda(\Gamma')} M$ is given by
   \[ f^\Gamma_{N} = f^{\Gamma'}_{M} \pmod{\Lambda(\Gamma)^\times}.\]
  \end{lem}

  \begin{proof} Immediate from the structure theorem, since $\Lambda(\Gamma)$ is free as a $\Lambda(\Gamma')$-module and hence flat.\end{proof}

 \section{The Selmer groups}
 
   Let $F$ be a finite extension of $\QQ$, $E$ an elliptic curve defined over $F$ and $p$ a prime $\geq 5$. Let $F_\infty$ be an admissible $p$-adic Lie extension of $F$ with Galois group $\Sigma$. We choose a set $S$ of primes of $F$ containing the primes dividing $p$, the primes where $E$ has bad reduction and the primes which ramify in $F_\infty/F$. 

   \begin{definition}
    For a finite extension $L$ of $F$ unramified outside $S$, and $v \in S$, we write
    \[ J_v(L) = \bigoplus_{q \mid v} H^1(L_q, E)(p).\]
    For an infinite extension (again unramifed outside $S$) we write 
    \[J_v(L) = \varinjlim_{L' \subset L} J_v(L'),\]
    where the direct limit is taken over finite extensions of $F$ contained in $L$.
   \end{definition}

   Recall that the Selmer group $\Sf$ is defined by the exact sequence
   \[
    0\rTo\Sf\rTo H^1(G_S(F_\infty),\Ep)\rTo^{\lambda_S(F_\infty)} \bigoplus_{v\in S}J_v(F_\infty).
   \]
   We define the large Selmer group $\Sp$ by the exact sequence
   \[
    0\rTo\Sp\rTo H^1(G_S(\Fc),\Ep)\rTo \bigoplus_{v\in S - S'}J_v(\Fc),
   \]
   where (as in the introduction) $S' \subset S$ denotes the set of primes of $F$ not dividing $p$ whose inertia group in $\Sigma$ is infinite. If $\Sc$ is $\Lambda(\Gamma)$-torsion, then as shown in \cite{coatesschneidersujatha03} the classical Selmer group $\Sc$ and $\Sp$ are related by the short exact sequence
   \begin{equation}\label{shortexactlarge2}
    0\rTo\Sc\rTo\Sp\rTo\bigoplus_{v\in S'}J_v(\Fc)\rTo 0.
   \end{equation}
  
\section{Local results}

 Let $H=\Gal(F_\infty/\Fc)$ and $\Gamma=\Gal(\Fc/F)$, as above. If $v$ is a prime of $F$ and $x$ is a prime of $\Fc$ above $v$, we will write $\Gamma_x$ for the decomposition group of $x$, and for $w$ a prime of $F_\infty$ above $x$, we write $H_w$ for the decomposition group of $w$ in $H$. Note that there will be finitely many primes of $\Fc$ above each prime of $F$, so $\Gamma_x$ will have finite index in $\Gamma$, but $H_w$ will not have finite index in $H$ in general. We identify $\Gamma_x$ and $H_w$ with the Galois groups $\Gal(\Fc_x / F_v)$ and $\Gal(\Fiw / \Fc_x)$ in the usual fashion. We let $\Sigma_w = \Gal(\Fiw/F_v)$, so $\Sigma_w / H_w \cong \Gamma_x$.

 For a prime $v$ of $F$ above $p$, we define the following condition:
 \begin{itemize}
  \item (R$_v$) $E$ has either good ordinary or split multiplicative reduction at $v$ 
 \end{itemize}
 
 \subsection{Local Cohomology}
 
 Let $v$ be a prime of $F$ in $S$, and choose compatible primes $x$ and $w$ of $\Fc$ and $F_\infty$ above $v$. Recall that we defined $S'$ to be the subset of $S$ of primes not dividing $p$ and having infinite inertia group in $\Sigma$.

  \begin{lem}\label{locdim1}
   If $v$ does not divide $p$, then $\Sigma_w$ has dimension $2$ if and only $v \in S'$. Otherwise, $\Sigma_w$ has dimension $1$.
  \end{lem}

  \begin{proof}
   Since $v \nmid p$ and $\Sigma$ is a $p$-adic Lie group, the wild inertia subgroup of $\Fiw$ is finite. The result now follows immediately from the well-known fact (c.f. \cite{serre72}) that the maximal tamely ramified extension of $F_v$ has topologically cyclic Galois group over the maximal unramified extension, and the maximal unramified $p$-extension of $F_v$ is the cyclotomic extension $\Fc_x$.
  \end{proof} 
   
   \begin{lem}\label{triviallocalterm}
    Let $v \in S'$. Then $J_v(F_\infty)=0$. 
   \end{lem}
   \begin{proof}
    By local Kummer theory, we have $H^1(F_{\infty,w},E)(p) \cong H^1(F_{\infty,w},E_{p^\infty})$. 
    But Lemma \ref{locdim1} shows that since $\dim(\Sigma_w)=2$, the profinite degree of the extension
    $\bar{F}_v$ over $F_{\infty,w}$ is coprime to $p$, so $\cd_p(\Gal(\bar{F}_v/F_{\infty,w}))=0$.   
   \end{proof}
  
   \begin{prop}
    For all $v \nmid p$, $J_v(\Fc)$ is a cofinitely generated $\ZZ_p$-module. 
   \end{prop}
   \begin{proof}
    For every prime $v$ of $F$, there are only finitely many primes $x$ of $\Fc$ dividing $v$, and we have
    \[ J_v(\Fc) = \bigoplus_{x \mid v} H^1(\Fc_x, E)(p) = \bigoplus_{x \mid v} H^1(\Fc_x, \Ep)\]
    where the second equality again follows from Kummer theory.

    If $L$ is the maximal tamely ramified $p$-extension of $\Fc_x$ and $\Pi = \Gal(L / \Fc_x)$, we have
    \[ 0 \rTo H^1(\Pi, \Ep(L)) \rTo H^1(\Fc_x,\Ep) \rTo H^1(L, \Ep)^{\Pi}\]
    by the Hochschild-Serre spectral sequence. But the last term is zero by the same argument as in lemma \ref{triviallocalterm}, and $H^1(\Pi, \Ep(L))$ is a cofinitely generated $\ZZ_p$-module.
   \end{proof}
   
   \begin{cor}\label{localcotorsion}
    For all $v \nmid p$, $J_v(\Fc)$ is $\Lambda(\Gamma)$-cotorsion. 
   \end{cor}
   
   \begin{lem}\label{Shapiro3}
   For any $v$ and all $i\geq 0$ there are canonical isomorphisms
   \[
    H^i(H,J_v(F_\infty))\cong \Ind_{\Gamma_x}^\Gamma H^i(H_w,H^1(\Fiw,E)(p)).
   \]
  \end{lem}

  \begin{proof}
   Easy consequence of Shapiro's lemma, since 
   \[J_v(F_\infty) = \Ind_{\Sigma_w}^\Sigma H^1(\Fiw,E)(p).\]
  \end{proof}
   
  \begin{lem}\label{localvanishing}
   If $v\in S'' = S \setminus S'$ and $v \nmid p$, then
   \[
    H^i(H,J_v(F_\infty))=0
   \]
   for all $i\geq 1$.
  \end{lem}

  \begin{proof}
   This follows from lemma \ref{Shapiro3} since for all such $v$ the group $H_w$ is finite, and as $F_\infty/F$ is admissible, its order must be prime to $p$.
  \end{proof}

  Now let $v$ be a prime of $F$ dividing $p$, and denote by $k_{\infty,w}$ 
  the residue field of $\Fiw$. If $E$ has good ordinary reduction at $v$, then denote by $\tilde{E}_v$ the
  reduction of $E$ modulo $v$. We define
  \[
   D_w=
   \begin{cases}
    \tEp(k_{\infty,w})&   \text{if $E$ has good ordinary reduction at $v$}\\
    \QQ_p/\ZZ_p&  \text{if $E$ has split multiplicative reduction at $v$}
   \end{cases}
  \]
  So $D_w$ is a $\Sigma_w$-module (with trivial action in the split multiplicative case), of finite $\ZZ_p$-corank.
  
  \begin{lem}\label{localisom3}
   Let $v$ be a prime of $F$ dividing $p$ and assume that (R$_v$) holds. Then for all $i\geq 1$, there 
   are canonical isomorphisms of $\Gamma_x$-modules
   \begin{equation}
    H^i(H_w,H^1(\Fiw,E)(p))\cong H^{i+2}(H_w,D_w).
   \end{equation} 
  \end{lem} 

  \begin{proof}
   We follow the strategy of \cite[Lemma 5.16]{coateshowson01}. The extension $\Fiw$ of $F_v$ is deeply ramified in the sense of \cite{coatesgreenberg96} since it contains the deeply ramified field $\Fc_x$, so by propositions 4.3 and 4.8 of {\it op.cit.} there is a canonical isomorphism of $\Sigma_w$-modules 
   \[ H^1(\Fiw,E)(p)\cong H^1(F_{\infty,w},D_w).\]
   Since $p^\infty$ divides the profinite degree of $F_\infty/F$, we have $\cd_p \Gal(\overline{F_v} / \Fiw) \le 1$, and hence $H^i(F_{\infty,w},D_w)=0$ for all $i\geq 2$. 

   By the Hochschild-Serre spectral sequence, for each $i \ge 1$ we have exact sequences
   \[ H^{i + 1}(\Fc_x, D_w) \to H^i(H_w, H^1(\Fiw, D_w)) \to H^{i+2}(H_w, D_w) \to H^{i+2}(\Fc_x, D_w).\]
   The two end terms both vanish, since $\cd_p \Gal(\overline{F_v} / \Fc_x)$ is also at most 1, and the result follows.
  \end{proof}
  
  The following proposition can be seen as the analogue of Condition 2 in \cite{zerbes04} and condition (Fin$_{\rm loc}$) in \cite{zerbes09}.

  \begin{prop}\label{localakashi3} Suppose $F_\infty / F$ is strongly admissible. Then for all $v \mid p$, we have $f^{\Sigma_w}_{D_w}\thicksim 1$.
  \end{prop}

  \begin{proof}
   By hypothesis, for all such $v$ the field $\Fiw$ contains both the cyclotomic $\ZZ_p$-extension and the unramified $\ZZ_p$-extension of $F_v$. Now these two extensions are disjoint (up to a finite extension of $F_v$ of order prime to $p$), so the result is a special case of Proposition \ref{akashi=1}.
  \end{proof}

 \subsection{Analysis of the local restriction maps}

  Let $v \in S''$. Applying lemma \ref{Shapiro3} with $i = 0$, we have
  \[ J_v(F_\infty)^H = \Ind_{\Gamma_x}^\Gamma H^1(\Fiw,E)(p)^{H_w}\]
  for any choice of primes $x$ and $w$ above $v$. There is a natural restriction map of $\Gamma_x$-modules
  \[ \gamma_x:H^1(\Fc_x,E)(p) \rightarrow H^1(\Fiw,E)(p)^{H_w},\]
  and hence a map of $\Gamma$-modules $\gamma_v: J_v(\Fc) \to J_v(F_\infty)^H$.
  
  \begin{lem}\label{localkercokerzero3}
   Let $x$ be a prime of $\Fc$ which divides a prime in $S''$ but does not divide $p$. Then both
   $\ker(\gamma_x)$ and $\coker(\gamma_x)$ are zero.
  \end{lem}

  \begin{proof}
   See \cite[Lemma 3.4]{zerbes09}.
  \end{proof}
  
  \begin{lem}\label{localkercoker3}
   Let $v$ be a prime of $F$ dividing $p$, and assume that (R$_v$) holds. Let $x$ be a prime of $\Fc$ 
   above $v$. Then 
   \begin{align*}
    \ker(\gamma_x)  &=H^1(H_w,D_w),\\
    \coker(\gamma_x)&=H^2(H_w,D_w).
   \end{align*}
  \end{lem}

  \begin{proof}
   Both $\Fiw$ and $\Fc_x$ are deeply ramified in the sense of \cite{coatesgreenberg96}, so we have 
   isomorphisms
   \begin{align*}
    H^1(\Fc_x,E)(p)&\cong H^1(\Fc_x,D_w),\\
    H^1(\Fiw,E)(p)&\cong H^1(\Fiw,D_w).
   \end{align*}
   The lemma now follows from the inflation-restriction exact sequence, since (as observed in the proof of lemma \ref{localisom3}) we have $H^2(\Fc_x, D_w) = 0$.
  \end{proof}

 \section{Global results}
 
  In this section, again $E$ is a elliptic curve defined over a number field $F$, $p$ a prime $\geq 5$ and
  $F_\infty/F$ an admissible $p$-adic Lie extension, with Galois group $\Sigma=\Gal(F_\infty/F)$. Let $H=\Gal(F_\infty/\Fc)$ and $\Gamma=\Gal(\Fc/F)$ as before. Define the following hypotheses:
  \begin{itemize}
   \item (R) $E$ has either good ordinary or split multiplicative reduction at the primes of $F$ dividing $p$,
   \item (Tors$_{\rm cyc}$) $\mathcal{C} (E/\Fc)$ is $\Lambda(\Gamma)$-torsion,
   \item (Tors$_\infty$) $\mathcal{C} (E/F_\infty)\in\MHS$.
  \end{itemize}
  We note that (Tors$_\infty$) implies (Tors$_{\rm cyc}$) \cite[lemma 2.6]{zerbes09}.

  \subsection{Global cohomology}
  
   We start with the following observation:
   
   \begin{prop}\label{selmercotorsion}
    $\Sp$ is $\Lambda(\Gamma)$-cotorsion if and only if $\Sc$ is, and if this is the case, then
    \[ f^\Gamma_{\Sp}=f^\Gamma_{\Sc}\times\prod_{v\in S'} f^\Gamma_{J_v(\Fc)}.\]
   \end{prop}

   \begin{proof}
    By definition, we have an exact sequence
    \[0\rTo\Sc\rTo\Sp\rTo\bigoplus_{v\in S'}J_v(\Fc) \]
    and the last term is a cotorsion module by proposition \ref{localcotorsion}, so $\Sf$ and $\Sc$ have the same $\Lambda(\Gamma)$-corank. Moreover, when $\Sc$ is cotorsion, the map onto $\bigoplus_{v\in S'}J_v(\Fc)$ is surjective (equation \eqref{shortexactlarge2} above) so the relation between the Akashi series follows from proposition \ref{multiplicativity3}.
   \end{proof}
   
   \begin{lem}\label{globalisom3}
    Suppose that (Tors$_{\infty}$) holds. Then for all $i\geq 1$, there are canonical isomorphisms
    \[
     H^i(H,H^1(G_S(F_\infty),\Ep))\cong H^{i+2}(H,\Ep(F_\infty)).
    \]
   \end{lem}

   \begin{proof} 
    We note that (Tors$_\infty$) implies that $H^i(G_S(F_\infty),\Ep) = 0$ for all $i \ge 2$; this is obvious for $i > 2$ as $\cd_p(G_S(F_\infty)) = 2$, and for $i = 2$ it is \cite[proposition 6]{zerbes04} (if $\Ep$ is rational over $F_\infty$) or \cite[theorem 4.5]{hachimorivenjakob03} (if $\Ep(F_\infty)$ is finite). Similarly, since (Tors$_\infty$) implies (Tors$_{\rm cyc}$), we have $H^i(G_S(\Fc),\Ep) = 0$ for all $i \ge 2$. The result now follows by comparing the $\Fc$ and $F_\infty$ cohomology of $\Ep$ using the Hochschild-Serre spectral sequence.
   \end{proof}

   The following proposition can be seen as the analogue of condition (EC$_{\rm glob}$) in \cite{zerbes04} and condition (Fin$_{\rm glob}$) in \cite{zerbes09}.

   \begin{prop}\label{globalakashi3}
    We have $f^\Sigma_{\Ep(F_\infty)}\thicksim 1$.
   \end{prop}  
   
   \begin{proof}
    If $\Ep$ is not rational over $F_\infty$, then $\Ep(F_\infty)$ is finite, and so $H^i(H,\Ep(F_\infty))$ is finite for all $i\geq 0$ by \cite[Proposition 5.4]{zerbes09} (the assumption in the statement of the proposition that $\operatorname{Lie} H$ be reductive is not used when the torsion is rational). This clearly implies that $f_{\Ep(F_\infty)}\thicksim 1$.

    Hence we may assume $F(\Ep) \subseteq F_\infty$. Then (up to a finite extension of $F$ of order prime to $p$) the Galois group of the extension $F(\Ep)$ of $F$ is of the form treated in Lemma \ref{akashi=1}.
   \end{proof}

  \subsection{The Akashi series of $\Sf^H$}

   The main tool in relating the $\Gamma$-Akashi series of $\Sf^H$ and $\Sp$ is the following commutative diagram appearing in \cite[\S 4.2]{zerbes09}, which we refer to as the ``fundamental diagram'':
   \begin{diagram}\label{exactdiag2}
    0  \to & \Sf^H         & \rTo & H^1(G_S(F_\infty),\Ep)^H & \rTo^{\psi_S(F_\infty)} &\bigoplus_{v\in S}J_v(F_\infty)^H  \\
            & \uTo_{\alpha} &      & \uTo_{\beta}             &                         &  \uTo_{\gamma=\bigoplus_{v\in S''} \gamma_v} \\
    0  \to & \Sp           & \rTo & H^1(G_S(\Fc),\Ep)        & \rTo^{\lambda_S(\Fc)}   & \bigoplus_{v\in S''}J_v(\Fc)      & \rTo & 0
   \end{diagram} 

   \begin{lem}\label{comparison3}
    Both $\coker(\gamma)/\coker(\psi_S(F_\infty))$ and 
    $\image\psi_S(F_\infty)/\image(\gamma)$ are torsion $\Lambda(\Gamma)$-modules and have the same characteristic element.
   \end{lem}
   \begin{proof}
    We have an isomorphism of $\Gamma$-modules
    \[
     \image(\psi_S(F_\infty))/\image(\gamma)
     \cong \coker(\gamma)/\coker(\psi_S(F_\infty)),
    \]
    so it is sufficient to show that they are torsion $\Lambda(\Gamma)$-modules. Now by Lemmas \ref{localkercokerzero3} and \ref{localkercoker3}, we have $\coker(\gamma)=\bigoplus_{x} H^2(H_w,D_w)$, where the sum runs over all primes of $\Fc$ dividing $p$, and $w$ is any choice of prime of $F_\infty$ above $x$. Since the set of such primes $x$ is finite, $\coker(\gamma)$ is a cofinitely generated $\ZZ_p$-module, so the same is true for $\coker(\psi_S(F_\infty))\subset\coker(\gamma)$.
   \end{proof}  
 
   \begin{prop}\label{chiGamma3}
    If (R) and (Tors$_{\rm cyc}$) hold, then 
    \[
     f^\Gamma_{\Sf^H}=f^\Gamma_{\Sel'(E/\Fc)}\times \delta,
    \]
    where 
    \[
     \delta=
     f_{\coker(\psi_S(F_\infty))}\times\frac{f^\Gamma_{H^2(H,\Ep(F_\infty))}}{f^\Gamma_{H^1(H,\Ep(F_\infty))}}\times
     \prod_{v\mid p}\frac{f^{\Gamma_x}_{H^1(H_w,D_w)}}{f^{\Gamma_x}_{H^2(H_w,D_w)}}.
    \]
   \end{prop}

   \begin{proof}
    Applying the snake lemma to the commutative diagram \eqref{exactdiag2} gives an exact sequence of $\Gamma$-modules
    \begin{multline}\label{snake3}
     0 \rTo \ker(\alpha) \rTo \ker(\beta) \rTo \ker(\gamma) \\
       \rTo \coker(\alpha) \rTo \coker(\beta) \rTo \image(\psi_S(F_\infty))/\image(\gamma)\rTo 0.
    \end{multline}
    The inflation-restriction exact sequence shows that
    \begin{align*}
     \ker(\beta)  &=H^1(H,\Ep(F_\infty)),\\
     \coker(\beta)&=H^2(H,\Ep(F_\infty)),
    \end{align*}
    which are cofinitely generated $\ZZ_p$-modules and hence $\Lambda(\Gamma)$-cotorsion. Similarly, it follows from Lemma \ref{localkercoker3} that 
    \begin{align*}
     \ker(\gamma)   &=\bigoplus_{v \mid p} \Ind_{\Gamma_x}^\Gamma H^1(H_w,D_w),\\
     \coker(\Gamma) &=\bigoplus_{v \mid p} \Ind_{\Gamma_x}^\Gamma H^2(H_w,D_w)
    \end{align*}
    and the characteristic elements of these are just $f^{\Gamma_x}_{H^1(H_w,D_w)}$ and $f^{\Gamma_x}_{H^2(H_w,D_w)}$ by lemma \ref{chareltinduced}. It therefore follows from \eqref{snake3} that both $\ker(\alpha)$ and $\coker(\alpha)$ are cotorsion $\Lambda(\Gamma)$-modules, and using the multiplicativity of Akashi series in exact sequences we deduce that
    \[
     \frac{f^\Gamma_{\ker(\alpha)}}{f^\Gamma_{\coker(\alpha)}}=
     \big(f^\Gamma_{\coker(\psi_S(F_\infty))}\big)^{-1}\times\frac{f^\Gamma_{H^1(H,\Ep(F_\infty))}}{f^\Gamma_{H^2(H,\Ep(F_\infty))}}
     \times\prod_{v\mid p}\frac{f^\Gamma_{H^2(H_w,D_w)}}{f^\Gamma_{H^1(H_w,D_w)}}.
    \]
    The result is now immediate from the exact sequence of $\Gamma$-modules
    \[
     0 \rTo \ker(\alpha) \rTo \Sp \rTo \Sf^H \rTo \coker(\alpha) \rTo 0
    \]
    and the assumption that $\Sc$ (and hence $\Sp$, by proposition \ref{selmercotorsion}) is $\Lambda(\Gamma)$-cotorsion.
   \end{proof}

  \subsection{Proof of the main theorem}
   
   As shown in \cite[Proposition 2.5]{zerbes09}, assumption (Tors$_\infty$) implies that the map $\lambda_S(F_\infty)$ is surjective, i.e. we have the short exact sequence
   \begin{equation}\label{Selmer3}
    0 \rTo \Sf \rTo H^1(G_S(F_\infty),\Ep) \rTo \bigoplus_{v\in S}J_v(F_\infty) \rTo 0.
   \end{equation}
   
   \begin{lem}\label{finiteHcohomology3}
    Assume that the conditions (R) and (Tors$_\infty$) hold. Then 
    \[ \prod_{i\geq 1} \left(f^\Gamma_{H^i(H,\Sf)}\right)^{(-1)^i} = 
      \left(f^\Gamma_{\coker(\psi_S(F_\infty))}\right)^{-1} 
     \times \prod_{i\geq 3} \left(\frac{ 
       f^\Gamma_{H^i(H,\Ep(F_\infty))} 
      }
      {
       \prod_{v\mid p} f^{\Gamma_x}_{H^i(H_w,D_w)}
      }\right)^{(-1)^i}.
    \]
   \end{lem}

   \begin{proof}
    Taking $H$-cohomology of \eqref{Selmer3}, we get the long exact sequence
    \begin{multline}\label{skin}
     0 \rTo \coker(\psi_S(F_\infty)) \rTo H^1(H,\Sf)\rTo H^1(H,H^1(G_S(F_\infty),\Ep))\\
       \rTo\dots \rTo \bigoplus_{v\in S} H^d(H,J_v(F_\infty)) \rTo 0,
    \end{multline}
    where $d=\dim(H)$, which is equal to the $p$-cohomological dimension of $H$ since $H$ 
    has no element of
    order $p$. Now as shown in Lemmas \ref{localisom3} and \ref{globalisom3}, for all $i\geq 1$ 
    we have canonical isomorphisms
    \[
     H^i(H,H^1(G_S(F_\infty),\Ep))\cong H^{i+2}(H,\Ep(F_\infty))
    \]
    and, for all $v\mid p$,
    \[
     H^i(H,J_v(F_\infty))\cong \Ind_{\Gamma_x}^\Gamma H^{i+2}(H_w,D_w),
    \]
    where $x$ and $w$ are any choice of primes of $\Fc$ and $F_\infty$ above $v$. Finally, for $v \in S''$ not dividing $p$, $H^i(H, J_v(F_\infty))$ is zero by lemma \ref{localvanishing}.

    Note that both $H^i(H,\Ep(F_\infty))$ and $H^i(H_w,D_w)$ are cofinitely generated as $\ZZ_p$-modules, so they are cotorsion $\Lambda(\Gamma)$-modules. Splitting up \eqref{skin} into short exact sequences, and using the multiplicativity of characteristic elements in short exact sequences (Lemma \ref{multiplicativity3}), the result is now a direct consequence of Lemma \ref{comparison3}.
   \end{proof}
    
   \begin{prop}\label{endofproof}
    If (R) and (Tors$_\infty$) hold, then 
    \[f^\Sigma_{\Sf}=T^r \times f^\Gamma_{\Sp},\]
    where $r$ is the number of primes of $F$ dividing $p$ where $E$ has split multiplicative reduction.
   \end{prop}

   \begin{proof}
    By the definition of the Akashi series, we have
    \[ 
     f^\Sigma_{\Sf} = f^\Gamma_{\Sf^H} \times \prod_{i \ge 1} \left(f^\Gamma_{H^i(H, \Sf)}\right)^{(-1)^i}.
    \]
    By Propositions \ref{chiGamma3} and \ref{finiteHcohomology3}, this is
    \[
     f^\Sigma_{\Sf} = f^\Gamma_{\Sp}\times\prod_{i\geq 1}\big(f^\Gamma_{H^i(H,\Ep(F_\infty))}\big)^{(-1)^i}\times
     \prod_{v\mid p}\prod_{i\geq 1}\big(f^{\Gamma_x}_{H^i(H_w,D_w)}\big)^{(-1)^{i+1}}.
    \]
    Since the Akashi series $f^\Sigma_{\Ep(F_\infty)}$ and $f^{\Sigma_w}_{D_w}$ are both 1, by propositions \ref{globalakashi3} and \ref{localakashi3} respectively, we deduce that
    \[
     f^\Sigma_{\Sf} = f^\Gamma_{\Sp}\times \left(f^\Gamma_{H^0(H, \Ep(F_\infty))}\right)^{-1} \times \prod_{v\mid p} f^{\Gamma_x}_{H^0(H_w,D_w)}.
    \]
    The group $H^0(H,\Ep(F_\infty))=\Ep(\Fc)$ is finite by Imai's theorem (c.f. \cite{imai75}), so $f^\Gamma_{\Ep(\Fc)}\thicksim
    1$. Similarly, if $v$ is a prime of $F$ dividing $p$ where $E$ has good ordinary reduction, then $H^0(H_w, D_w) = \tEp(k^{\rm cyc}_x)$ is finite and $f^{\Gamma_x}_{H^0(H_w, D_w)}\thicksim 1$. 

    Now let $v$ be a prime of $F$ dividing $p$ where $E$ has split multiplicative
    reduction. Then $H^0(H_w,D_w)\cong \QQ_p/\ZZ_p$, so $f^{\Gamma_x}_{H^0(H_w,D_w)}= T$.
   \end{proof}

   Combining this with proposition \ref{selmercotorsion} completes the proof of Theorem \ref{maintheorem}.

\section{The field of $p$-division points}\label{gl2ext}

  In the proof of theorem \ref{maintheorem}, the assumption that $F_\infty$ is strongly admissible is vital, since it guarantees that the Akashi series of the local terms are equal to 1. In this section, we consider a particularly important admissible extension, the extension $F_\infty = F(\Ep)$ generated by the $p$-power torsion points of $E$. If $E$ does not have complex multiplication, the main result of~\cite{serre72} shows that
  $\Sigma=\Gal(F_\infty/F)$ is an open subgroup of $\operatorname{GL}_2(\ZZ_p)$, and if $p \ge 5$, then $\Gal(F_\infty/F)$ has no element of order $p$, so $F_\infty$ is admissible. 

  The decomposition subgroups of $\Sigma$ have been explicitly determined \cite[lemma 5.1]{coateshowson01}. We briefly recall the results:

  \begin{lem}
   Let $E$ be an elliptic curve defined over a number field $F$ without complex multiplication, $p$ a prime $\geq 5$, $F_\infty=F(\Ep)$ and $\Sigma=\Gal(F_\infty/F)$. Let $v$ be a prime of $F$ and $w$ a prime of $F_\infty$ above $v$.
   \begin{enumerate}
    \item If $v \nmid p$, then the inertia subgroup of $\Sigma_w$ is infinite if and only if $E$ has bad but not potentially good reduction at $v$ (or equivalently if $j(E)$ is non-integral at $v$).
    \item If $v \mid p$ and $E$ has good ordinary reduction at $v$, then $F_{\infty, w}$ contains the unramified $\ZZ_p$-extension of $F_v$.
    \item If $v \mid p$ and $E$ has split multiplicative reduction at $v$, then $F_{\infty, w} = F(\mu_{p^\infty}, \sqrt[p^\infty]{q_{v}(E)})$ where $q_v(E)$ is the Tate period of $E$. In particular, $F_{\infty, w}$ has no infinite Galois sub-extensions linearly disjoint from the cyclotomic $\ZZ_p$-extension of $F_v$.
   \end{enumerate}
  \end{lem}

  It follows that if $E$ has split multiplicative reduction at some primes above $p$, $F_\infty / F$ is not strongly admissible, so theorem \ref{maintheorem} does not apply. Nonetheless, when this is not the case, we can still calculate the ``error term'' that arises: 

  \begin{thm}\label{theorem11}
   Let $E$ be an elliptic curve defined over a number field $F$, $p$ a prime $\geq 5$, $F_\infty=F(\Ep)$ and $\Sigma=\Gal(F_\infty/F)$. Assume that 
   \begin{enumerate}
    \item $E$ has either good ordinary or split multiplicative reduction at all primes of $F$ dividing $p$,
    \item $\mathcal{C} (E/F_\infty)\in\MHS$,
    \item $E$ does not have complex multiplication.
   \end{enumerate}
   Then the $\Sigma$-Akashi series of $\Sf$ is given by the following formula:
   \[ f^\Sigma_{\Sf}=f^\Gamma_{\Sc}\times\prod_{v\in\mathfrak{M}}f^\Gamma_{J_v(\Fc)}\times\prod_{v\in\mathfrak{R}} \left(T + 1 - \chi(\gamma_v)\right) \mod \Lambda(\Gamma)^*.\]
   Here, $\mathfrak{M}$ is the set of primes of $F$ not dividing $p$ where $E$ has non-integral $j$-invariant; $\mathfrak{R}$ is the set of primes of $F$ dividing $p$ where $E$ has split multiplicative reduction; and $\gamma_v$ is a generator of the decomposition group of a prime of $\Fc$ above $v$.
  \end{thm}
  
  \begin{proof}
   Let $v$ be a prime of $F$ above $p$, and let $D_w$ be as defined above. If $E$ has good ordinary reduction at $v$, then proposition \ref{akashi=1} applies to $F_{\infty, w} / F_v$, so we have $f^{\Sigma_w}_{D_w} = 1$. On the other hand, if $F$ has split multiplicative reduction at $p$, we have $D_w = \QQ_p / \ZZ_p$ and $\Sigma_w = \ZZ_p \rtimes \Gamma_x$, with $\Gamma_x$ acting on $\ZZ_p$ via the cyclotomic character. Hence:
   \begin{align*}
    H^1(H_w, D_w) &= (\QQ_p / \ZZ_p)(\chi) \\
    H^i(H_w, D_w) &= 0 \quad \text{for all $i \ge 2$}.
   \end{align*}
   The result now follows exactly as in proposition \ref{endofproof} above.
  \end{proof}

  \begin{cor}
   Under the conditions of \ref{theorem11}, if $\Sf$ and $\Sc$ have finite generalised $\Sigma$- and $\Gamma$-Euler characteristics, then these are related by the formula 
   \[
    \chi(\Sigma,\Sf)=\chi(\Gamma,\Sc)\times\left| \prod_{v\in\mathfrak{M}}L_v(E,1)\times\prod_{v\in\mathfrak{R}}
    \#\mu_{p^\infty}(F_v)\right|_p.
   \]
  \end{cor}
  
  \begin{remark}
   Using the methods of \cite{zerbes09}, one can show that if $\Sha(E/F)(p)$ is finite, then $\Sf$ has finite generalised $\Sigma$-Euler characteristic if and only if and $\Sc$ has finite generalised $\Gamma$-Euler characteristic.
  \end{remark}

\subsection*{Acknowledgements} I would like to thank John Coates and David Loeffler for their interest, and the latter for his love and support (and his careful reading of the manuscript).


\end{document}